\documentclass[12pt,reqno]{article}

\usepackage[usenames]{color}
\usepackage{amssymb}
\usepackage{amsmath}
\usepackage{amsthm}
\usepackage{amsfonts}
\usepackage{amscd}
\usepackage{graphicx}
\usepackage{mathtools}
\usepackage{dirtytalk}
\usepackage{enumerate} 
\usepackage{tikz}

\usepackage[colorlinks=true,
linkcolor=webgreen,
filecolor=webbrown,
citecolor=webgreen]{hyperref}

\definecolor{webgreen}{rgb}{0,.5,0}
\definecolor{webbrown}{rgb}{.6,0,0}

\usepackage{color}
\usepackage{fullpage}
\usepackage{float}

\usepackage{graphics}
\usepackage{latexsym}
\usepackage{epsf}
\usepackage{breakurl}

\newcommand{\seqnum}[1]{\href{https://oeis.org/#1}{\rm \underline{#1}}}

\DeclarePairedDelimiter\floor{\lfloor}{\rfloor}
\DeclarePairedDelimiter\abs{\lvert}{\rvert}

\newcommand{\Z}{\mathbb{Z}}
\newcommand{\N}{\mathbb{N}}

\newcommand{\C}{\mathbb{C}}

\renewcommand{\phi}{\varphi}

\newcommand{\suchthat}{\colon}
\newcommand{\divides}{\mid}
\newcommand{\modd}[2]{#1\ \mbox{\rm (mod}\ #2\mbox{\rm )}}

\theoremstyle{plain}
    \newtheorem{theorem}{Theorem}
    \newtheorem{corollary}[theorem]{Corollary}
    \newtheorem{lemma}[theorem]{Lemma}

\theoremstyle{definition}

    \newtheorem{conjecture}[theorem]{Conjecture}

\theoremstyle{remark}

\newcommand\blfootnote[1]{%
  \begingroup
  \renewcommand\thefootnote{}\footnote{#1}%
  \addtocounter{footnote}{-1}%
  \endgroup
}

\author{Max A. Alekseyev \\
Department of Mathematics\\
The George Washington University\\
800 22nd Street NW \#7685\\
Washington, DC 20052 \\
USA\\
\href{mailto:maxal@gwu.edu}{\tt maxal@gwu.edu} \\
\and 
Tewodros Amdeberhan \\
Department of Mathematics \\
Tulane University \\
6823 St. Charles Avenue\\
New Orleans, LA 70118-5698\\
USA \\
\href{mailto:tamdeber@tulane.edu}{\tt tamdeber@tulane.edu} \\
\and 
Jeffrey Shallit$^*$
and Ingrid Vukusic$^*$\\
School of Computer Science\\
University of Waterloo\\
Waterloo, ON  N2L 3G1 \\
Canada\\
\href{mailto:shallit@uwaterloo.ca}{\tt shallit@uwaterloo.ca}\\
\href{mailto:ingrid.vukusic@uwaterloo.ca}{\tt ingrid.vukusic@uwaterloo.ca}\\
\blfootnote{Research of the last two authors supported by NSERC grant 2024-03725.}
}

\title{On the $p$-adic valuations of values of Legendre polynomials}

\begin{document}

\maketitle

\begin{abstract}
We prove an explicit formula for the $p$-adic valuation of the Legendre polynomials $P_n(x)$ evaluated at a prime $p$, and generalize an old conjecture of the third author. We also solve a problem proposed by Cigler in 2017.
\end{abstract}

\section{Introduction}

For a positive integer $n$ and a prime number $p$, let $\nu_p (n) := \max \{i \suchthat p^i \divides n\}$ denote the \emph{$p$-adic valuation} of $n$.  More generally, for a positive rational number $m/n$, we have $\nu_p(m/n) := \nu_p(m) - \nu_p(n)$, which further extends to all rational numbers $r$ via $\nu_p(r) := \nu_p(|r|)$ and $\nu_p(0):=+\infty$.

Determining the $p$-adic valuation of the elements of
various combinatorial sequences is an old, interesting, and often challenging problem.  
For example, in 1830 Legendre~\cite[p.~10]{Legendre:1830}
gave a celebrated formula for $\nu_p(n!)$:
\begin{equation}
\nu_p(n!) = \sum_{i \geq 1} \left\lfloor {\frac{n}{p^i}} \right\rfloor .
\label{leg1}
\end{equation}
Although this sum is formally over infinitely many values of $i$, its terms are $0$ for all sufficiently large $i$.
An alternative formulation is
\begin{equation}
\nu_p(n!) = \frac{n-s_p(n)}{p-1},
\label{leg2}
\end{equation}
where $s_p(n)$ denotes the sum of the base-$p$ digits of $n$.

Later, Kummer~\cite{Kummer:1852}
gave a formula for $\nu_p({n \choose  k })$;
namely, he expressed it as the number of carries
in the base-$p$ addition of $k$ and $n-k$.

Since then, the $p$-adic valuations of many other sequences have been studied. Among them are the Fibonacci and tribonacci numbers studied by Lengyel~\cite{Lengyel:1995} and by Marques and Lengyel~\cite{Marques&Lengyel:2014}, respectively.

In some instances, the $p$-adic valuation of a sequence $(c(n))_{n \geq 0}$ is \emph{$p$-regular}, that is, the $p$-kernel of the sequence $(\nu_p(c(n)))_{n \geq 0}$ produces a finitely generated
module~\cite{Allouche&Shallit:1992,Allouche&Shallit:2003a}. In other words, the set of subsequences
$$\{ (\nu_p ( c(p^e n + i)))_{n \geq 0} \suchthat e\geq 0, 0\leq i < p^e \}$$ is of finite rank over the rationals. 
For example, the sequence $c(n) = {{2n}\choose n}$ is $p$-regular, since from Eq.~\eqref{leg2} it follows that
the $p$-kernel
of $(\nu_p(c(n)))_{n \geq 0}$ is spanned by the three sequences $(\nu_p(c(n)))_{n \geq 0}$,
$(\nu_p (c(pn+p-1)))_{n \geq 0}$, and the constant sequence $1$.

Boros, Moll, and the third author studied the $2$-regularity of the $2$-adic valuation of certain polynomials associated with definite integrals~\cite{Boros&Moll&Shallit:2000}. 
Bell~\cite{Bell:2007} and Medina, Moll, and Rowland~\cite{Medina&Moll&Rowland:2017} studied the case of polynomial $c(n)$ more generally. To name a few other papers,
Shu and Yao \cite{Shu&Yao:2011} characterized analytic functions  $f \colon \Z_p \to \C_p$
without roots in $\N$ such that $(\nu_p (f(n)))_{n \geq 0}$ is $p$-regular. 
Medina and Rowland~\cite{Medina&Rowland:2015} further studied the $p$-regularity of the Fibonacci numbers, and Murru and Sanna~\cite{Murru&Sanna:2018} analyzed that of the more general Lucas sequences. 

In other cases, the $p$-adic valuations exhibit various kinds of regularities without actually being $p$-regular. For example, see~\cite{Aidagulov&Alekseyev:2018,Almodovar:2019,Givens&Moll:2020}.

In this paper, we study the $p$-adic valuations of the Legendre polynomials $P_n(x)$ evaluated at a prime number $p$, and show that they are $p$-regular.
In the next section, we provide some motivation for studying this question.


\section{Motivation}

Back in 1988, when the third author (JS) was an assistant editor of the problems section of
the {\it American Mathematical Monthly}, he received a submission from Nicholas Strauss and Derek Hacon with a proof of an inequality about the $3$-adic valuation of the sequence
$d(n) := \sum_{0 \leq i < n} {{2i} \choose i}$. The sequence $d(n)$ is present in the On-Line Encyclopedia of Integer Sequences (OEIS)~\cite{oeis} as sequence \seqnum{A006134}.

JS guessed that $A(n) := (\nu_3 (d(n)))$ might be a $3$-regular
sequence and used a computer program to discover the following heuristic relations:
\begin{align*}
A(3n+2) &= A(n)+2\\
A(9n) &= A(3n) \\
A(9n+1) &= A(3n) + 1 \\
A(9n+3) &= A(3n) \\
A(9n+4) &= A(3n+1) + 1 \\
A(9n+6) &= A(3n+1) \\
A(9n+7)  &= A(3n+1) + 1.
\end{align*}
They led JS to conjecture that
$$ \nu_3 (d(n)) = \nu_3 ({{2n} \choose n}) + 2 \nu_3 (n),$$
which he was later able to prove with the helpful advice of Jean-Paul Allouche.
The original problem proposal, modified to give the exact formula, eventually appeared as Problem 6625 in the {\it American Mathematical Monthly} in 1990~\cite{Strauss&Shallit:1990}, and a completely different solution by Don Zagier, based on $3$-adic analysis, was published two years later~\cite{Zagier:1992}.  The sequence $\nu_3(d(n))$ is the sequence \seqnum{A082490} in the OEIS.

After this, JS was inspired to use the same computer program to explore
whether the $p$-adic valuations of other combinatorial sums might (conjecturally) have similar identities. This effort was largely unsuccessful, with the exception of the following sequence:
Define
\begin{equation}
a(n) := \sum_{0 \leq i \leq n} {n\choose i}{{n+i} \choose i},
\label{leg}
\end{equation}
known as {\it central Delannoy numbers} and listed as sequence \seqnum{A001850} in OEIS. Define $b(n) := \nu_3(a(n))$; this is sequence
\seqnum{A358360} in the OEIS. Then the numerical evidence supported the following conjecture.
\begin{conjecture}
The sequence $(b(i))_{i\geq 0}$ satisfies the following identities:
$$
b(i) = \begin{cases} b(\lfloor i/3 \rfloor) + (\lfloor i/3 \rfloor \bmod 2), & \text{if 
$i \equiv \modd{0,2} {3}$;} \\
b(\lfloor i/9 \rfloor) + 1, &\text{if 
$i \equiv \modd{1} {3}$.}
\end{cases}
$$
\label{jeffcon}
\end{conjecture}
This conjecture appeared in \cite[p.~453]{Allouche&Shallit:2003b}, and JS frequently mentioned it in his talks (e.g.,~\cite{Shallit:2000}).
This conjecture was proved only in 2023 by Shen~\cite{Shen:2023}. 

In 2017, JS posted the conjecture and its generalization to arbitrary prime $p$ as a query on the MathOverflow website~\cite{Shallit:2017}, which ultimately inspired the authors to form a team and settle the generalized conjecture (Theorem~\ref{thm:main} below) with a joint effort. Namely, we prove explicit and recurrence formulae for $\nu_p(P_n(p))$ for every prime $p\geq 3$, from which Conjecture~\ref{jeffcon} follows as a partial case with $p=3$ and $a(n)=P_n(3)$.

As we will see in Theorem~\ref{cigler}, our results also imply the identity $\nu_p(M_n(p)) = \nu_p(P_n(p))$ conjectured by Cigler~\cite{Cigler:2017} for the polynomials
\begin{equation}\label{eq:Cigler_def}
    M_n(x) := \sum_{k=0}^n {n \choose k}^2 (x-1)^k.
\end{equation}

Along the lines of these results, we also pose the following open question.

\begin{conjecture} For every integer $n\geq0$, we have
\begin{align*}
\nu_3(\sum_{k=0}^n\binom{n}k^3 2^k) & = \begin{cases}
s_3(\frac{n-1}2)+1, & \text{if $n\equiv-1\pmod 6$}; 
\\ s_3(\lfloor\frac{n+1}2\rfloor), & \text{otherwise}. 
\end{cases} 
\end{align*}
\end{conjecture}


\section{Main results}

In this section, we state the main results of the paper.

\begin{theorem}\label{thm:vp-Legendre_r}
Let $p$ be a prime number and $r$ be a rational number such that $\nu_p(r)\geq 1$. Then for every integer $n\geq 0$, we have
$$
\nu_p(P_n(r)) =
\begin{cases}
    \nu_p({n \choose n/2}), & \text{if $n$ is even and $ p\geq 3$}; \\
    \nu_p({n-1 \choose (n-1)/2}) + \nu_p(r) + \nu_p(n),  & \text{if $n$ is odd and $ p\geq 3$}; \\
    \nu_2({n \choose n/2})-n, & \text{if $n$ is even and $ p=2$}; \\
    \nu_2({n-1 \choose (n-1)/2}) + \nu_2(r) + 1 - n, & \text{if $n$ is odd and $p=2$}.
\end{cases}
$$
Equivalently,
$$\nu_p(P_n(r)) = \nu_p(\frac1{2^n}\binom{n}{\lfloor n/2\rfloor}) + (n\bmod 2) \nu_p(r(n+1)).$$
\end{theorem}

\begin{theorem}\label{thm:vp-Legendre_explicit_p3}
Let $p \geq 3$ be a prime number. Then for every integer $m\geq 0$, we have
\begin{align*}
    \nu_p(P_{2m}(p)) &= \nu_p({2m \choose m}); \\
    \nu_p(P_{2m+1}(p)) &= 1 + \nu_p(2m+1) + \nu_p({2m \choose m}).
\end{align*}
Moreover, for every integer $n\geq0$,
$$
\nu_p(P_n(p)) = \frac{2s_p(\lfloor n/2\rfloor) - s_p(n) + (n\bmod 2)p}{p-1}.
$$
\end{theorem}

\begin{theorem}\label{thm:vp-Legendre_explicit_p2}
For every integer $n \geq 0$, we have
\[
    \nu_2(P_n(2)) = (n \bmod 2) - \nu_2(n!).
\]
\end{theorem}

\begin{theorem}\label{thm:main}
Let $p \geq 3$ be a prime number and $f(n) := \nu_p(P_n(p))$.
Then for all integers $n\geq 0$ and $0\leq a < p$, we have
\[
    f(pn + a) = \begin{cases}
       f(n) + (n \bmod 2), & \text{if } a \text{ is even};\\
       f(n) + 1 - (n \bmod 2), & \text{if } a \text{ is odd}.
    \end{cases}
\]    
\end{theorem}

For the special case $p = 3$, with the aid of Eq.~\eqref{useful} for $x = 3$, Theorem~\ref{thm:main} implies Conjecture~\ref{jeffcon} by setting $i=3n+a$ with $0\leq a<3$. The case of $a \in \{0,2\}$ is immediate, while for $a=1$, we get $b(i) = b(n) + 1 - (n\bmod 2)$. The case $a=1$ then follows by writing $n = 3m + b$ and noticing that $n\bmod 2 = m \bmod 2$ when $b \in \{0,2\}$,
while $(m \bmod 2) + (n \bmod 2)=1$ when $b=1$.
We thus recover the recent result of Shen~\cite{Shen:2023}. 

As a corollary of Theorems~\ref{thm:vp-Legendre_explicit_p2} and \ref{thm:main} we also get
\begin{corollary}
    For every prime number $p$, the sequence $(\nu_p(P_n(p)))_{n \geq 0}$ is $p$-regular.
\end{corollary}

\begin{theorem}\label{cigler}
Let $M_n(x)$ be defined as in Eq.~\eqref{eq:Cigler_def}.
Then for every integer $n \geq 0$ and prime $p\geq 3$, we have
$$\nu_p(M_n(p)) = \nu_p(P_n(p)).$$
\end{theorem}

\section{Background and preliminary results}


There exist many formulae for Legendre polynomials $P_n(x)$, including the following identity~\cite{Koepf2014}:
\begin{equation} \label{useful}
    P_n(x) = \sum_{0\leq k\leq n} {n \choose k} {{n+k} \choose k} \left(\frac{x-1}2\right)^k,
\end{equation}
which is \cite[\S 93, p.~166, Eq.~(2)]{Rainville:1960} (expressed there in the form of hypergeometric series).
The following two formulas for Legendre polynomials are due to Rodrigues~\cite{Rodrigues:1816}:
\begin{align}
\label{eq:Rodrigues}
    P_n(x) &= \frac{1}{2^n} \sum_{k=0}^{\floor{n/2}} (-1)^k {n \choose k} {2n - 2k \choose n} x^{n-2k};\\
\label{eq:Rodrigues_Cigler}
    P_n(x) &= \frac{1}{2^n} \sum_{k=0}^n {n \choose k}^2 (x-1)^k (x+1)^{n-k}.
\end{align}

Let $p$ be a prime.
As before, let $s_p(n)$ denote the sum of the base-$p$ digits of $n$. 
From Eq.~\eqref{leg2}, for integers $n\geq k\geq 0$ we easily obtain
\begin{equation}\label{eq:Kummer_nnover2}
    \nu_p(\binom{n}{k}) = \nu_p(\frac{n!}{k!(n-k)!}) = \frac{s_p(k) + s_p(n-k) - s_p(n)}{p-1}.
\end{equation}
Moreover, for $n \geq 0$ and $0 \leq a < p$, we have
\begin{equation}\label{eq:Sp_pna}
    s_p(np + a) = s_p(n) + a.
\end{equation}

The following lemma, together with Eq.~\eqref{eq:Kummer_nnover2}, will be a key to expressing $\nu_p(P_n(p))$ in terms of $s_p$.

\begin{lemma}\label{lem:v_odd}
For every prime $p$ and every integer $m\geq 0$,
$$
\nu_p(2m+1) + \nu_p(\binom{2m}{m}) = \nu_p(m+1) + \nu_p(\binom{2m+1}{m}) = \frac{2s_p(m)-s_p(2m+1)+1}{p-1}.
$$
\end{lemma}

\begin{proof} The first equality follows from the identity $(2m+1)\binom{2m}{m}=(m+1)\binom{2m+1}{m}$. To prove the second equality, we consider two cases depending on whether $p$ divides $2m+1$.

When $\nu_p(2m+1)=0$, we have $s_p(2m+1) = s_p(2m)+1$ and use Eq.~\eqref{eq:Kummer_nnover2} to get
$$
\nu_p(2m+1) + \nu_p(\binom{2m}{m}) =
\frac{2s_p(m)-s_p(2m)}{p-1} = \frac{2s_p(m)-s_p(2m+1)+1}{p-1}.
$$

When $\nu_p(2m+1)\geq 1$, we have $\nu_p(m+1)=0$ and thus $s_p(m+1)=s_p(m)+1$, implying that
$$
\nu_p(m+1) + \nu_p(\binom{2m+1}{m}) 
= \frac{s_p(m)+s_p(m+1)-s_p(2m+1)}{p-1} = \frac{2s_p(m)-s_p(2m+1)+1}{p-1}.
$$
\end{proof}

In view of formula~\eqref{eq:Rodrigues}, let us define
\begin{equation}\label{eq:def_Q}
    Q_n(x) 
    := 2^n P_n(x) 
    = \sum_{k=0}^{\floor{n/2}} (-1)^k {n \choose k} {2n - 2k \choose n} x^{n-2k}.
\end{equation}

\begin{lemma}\label{lem:main_even}
Let $p$ be a prime number and $r$ be a rational number such that $\nu_p(r) \geq 1$. Then for every integer $m\geq 0$, we have
\begin{align*}
    \nu_p(Q_{2m}(r)) &= \nu_p({2m \choose m}).
\end{align*}
\end{lemma}

\begin{proof}
Since $Q_0(r) = 1$, the lemma statement is clearly true for $m = 0$. 
From now on, let $m\geq 1$.
By the definition \eqref{eq:def_Q} of $Q_n(x)$, we have
\[
    Q_{2m}(x) = \sum_{i = 0}^{2m} a_i x^i,
\]
where
\[
    a_i := \begin{cases}
        (-1)^{m-i/2}{2m \choose m-i/2} {2m + i \choose 2m}, & \text{if $i$ is even};\\
        0, & \text{if $i$ is odd}.
    \end{cases}
\]
In particular, we have $a_0 = (-1)^m {2m \choose m}$, and so
\begin{equation}\label{eq:vp_a0p}
    \nu_p(a_0) 
    = \nu_p({2m \choose m}).
\end{equation}
Now our goal is to show 
\begin{equation}\label{eq:wts_ai_p_even}
    \nu_p(a_i r^i) > \nu_p(a_0) 
    \quad \text{for all even } 2 \leq i \leq 2m.
\end{equation}
Let $i = 2j$.
Equivalently to \eqref{eq:wts_ai_p_even}, we need to show that
\begin{equation*}
    \nu_p (a_{2j}/ a_0) \geq - 2j \cdot \nu_p(r) + 1
    \quad \text{for all } 1 \leq j \leq m.
\end{equation*}
Since $\nu_p(r)\geq 1$, it suffices to show that
\begin{equation}\label{eq:wts_ai_p_2_even}
    \nu_p (a_{2j}/ a_0) \geq - 2j + 1
    \quad \text{for all } 1 \leq j \leq m.
\end{equation}
Expanding $a_{2j}/ a_0$, and regrouping terms, we obtain
\begin{align*}
    \abs{a_{2j}/ a_0} 
    &= {2m \choose m-j} {2m + 2j \choose 2m} \cdot {2m \choose m}^{-1}\\
    &= \frac{(2m)!}{(m-j)! (m+j)!} \cdot \frac{(2m+2j)!}{(2m)! (2j)!} 
        \cdot \frac{m! m!}{(2m)!}\\[.1in]
    &= \frac{(2m+2j)!m!}{(2m)!(m+j)!} \cdot \frac{m!}{(m-j)!} \cdot \frac{1}{(2j)!}\\
    &= \frac{\prod_{\ell=1}^{2j}(2m+\ell)}{\prod_{\ell=1}^j(m+\ell)} \cdot \frac{m!}{(m-j)!} \cdot \frac{1}{(2j)!}.
\end{align*}
It is easy to see that $\prod_{\ell=1}^j(m+\ell)=\frac1{2^j}\prod_{\ell=1}^j(2m+2\ell)$ divides $\prod_{\ell=1}^{2j}(2m+\ell)$, implying that 
$a_{2j}/ a_0 = b/(2j)!$ for some non-zero integer $b$.
By Legendre's formula \eqref{leg2} we have
\[
    \nu_p((2j)!)
    = \frac{2j - s_p(2j)}{p-1}
    \leq 2j-1.
\]
Thus,
    $\nu_p(a_{2j}/ a_0) = \nu_p(b) - \nu_p((2j)!) \geq -2j + 1$,
and so inequalities \eqref{eq:wts_ai_p_2_even} and \eqref{eq:wts_ai_p_even} hold, implying that $\nu_p(Q_{2m}(r)) = \nu_p(a_0)$, which together with Eq.~\eqref{eq:vp_a0p} completes the proof.
\end{proof}

\begin{lemma}\label{lem:main_odd}
Let $p$ be a prime number,  $m\geq 0$ be an integer, and $r$ be a rational number such that $\nu_p(r) \geq 1$. When $p\geq 3$, we have
\begin{align*}
    \nu_p(Q_{2m+1}(r)) &= \nu_p(r) + \nu_p(2m+1) + \nu_p({2m \choose m}),
\end{align*}
and for $p = 2$,
\[
    \nu_2(Q_{2m+1}(r)) = 1 + \nu_2(r) + \nu_2({2m \choose m}).
\]
\end{lemma}

\begin{proof}
The proof idea is similar to that of Lemma~\ref{lem:main_even}. 
Since $Q_1(r) = 2r$, the statement of the lemma is clearly true for $m = 0$. From now on, let $m\geq 1$.
By the definition \eqref{eq:def_Q}, we have
\[
    Q_{2m+1}(x) = \sum_{i = 0}^{2m+1} a_i x^i,
\]
where
\[
    a_i = \begin{cases}
        0, & \text{if $i$ is even};\\
        (-1)^{m-(i-1)/2} {2m+1 \choose m-(i-1)/2} {2m+1 + i \choose 2m+1}, & \text{if $i$ is odd}.
    \end{cases}
\]
In particular, we have $a_1=2 (2m+1) {2m \choose m}$, and so
\begin{equation}\label{eq:vp_a1p}
    \nu_p(a_1 r) 
    = \begin{cases}
        \nu_p(r) + \nu_p(2m+1) + \nu_p({2m \choose m}), & \text{if } p \geq 3;\\
        1 + \nu_p(r) + \nu_p({2m \choose m}), & \text{if } p = 2.
    \end{cases}
\end{equation}
Now our goal is to show 
\begin{equation}\label{eq:wts_ai_p}
    \nu_p(a_i r^i) > \nu_p(a_1 r) 
    \quad \text{for all odd } i,\ 3\leq i \leq 2m+1.
\end{equation}
Let $i = 2j + 1$.
Equivalently to \eqref{eq:wts_ai_p}, we need to show
\begin{equation*}
    \nu_p (a_{2j+1}/ a_1) \geq - 2j \cdot \nu_p(r) + 1
    \quad \text{for all } j,\ 1 \leq j \leq m.
\end{equation*}
Since $\nu_p(r)\geq 1$, it suffices to show that
\begin{equation}\label{eq:wts_ai_p_2}
    \nu_p (a_{2j+1}/ a_1) \geq - 2j + 1
    \quad \text{for all } j,\ 1 \leq j \leq m.
\end{equation}
Expanding $a_{2j+1}/ a_1$, and regrouping terms, we obtain
\begin{align*}
    \abs{a_{2j+1}/ a_1} 
    &= {2m+1 \choose m-j} {2m + 2j + 2 \choose 2m+1} \cdot \left( 2 (2m+1) {2m \choose m} \right)^{-1}\\
    &= \frac{(2m+1)!}{(m-j)! (m+j+1)!} \cdot \frac{(2m+2j+2)!}{(2m+1)! (2j+1)!} 
        \cdot \frac{1}{2(2m+1)} \cdot \frac{m! m!}{(2m)!}\\[.1in]
    &= \frac{(2m+2j+2)!m!}{2(2m+1)!(m+j+1)!} \cdot \frac{m!}{(m-j)!} \cdot \frac{1}{(2j+1)!}\\
    &= \frac{\prod_{\ell=1}^{2j+2}(2m+\ell)}{2\prod_{\ell=1}^{j+1}(m+\ell)} \cdot \frac{m!}{(m-j)!} \cdot \frac{1}{(2j+1)!}.
\end{align*}
Since $2\prod_{\ell=1}^{j+1}(m+\ell)=\frac1{2^j}\prod_{\ell=1}^{j+1}(2m+2\ell)$ divides $\prod_{\ell=1}^{2j+2}(2m+\ell)$, it follows that $a_{2j+1}/ a_1 = b/(2j+1)!$ for some non-zero integer $b$.
By Legendre's formula~\eqref{leg2}, we have
\[
    \nu_p((2j+1)!)
    = \frac{2j+1 - s_p(2j+1)}{p-1}
    \leq 2j-1,
\]
where we used that $s_p(2j+1) \geq 1$ and $p-1\geq 2$ when $p\geq 3$, and $s_2(2j+1) \geq 2$ when $p=2$.
Thus, 
$\nu_p(a_{2j+1}/ a_1) = \nu_(b) - \nu_p((2j+1)!) \geq -2j + 1$,
and so inequalities \eqref{eq:wts_ai_p_2} and \eqref{eq:wts_ai_p} hold, 
implying that $\nu_p(Q_{2m+1}(r)) = \nu_p(a_1 r)$, which together with \eqref{eq:vp_a1p} completes the proof.
\end{proof}


\section{Proofs of the main results}
\label{proofs}

Theorem~\ref{thm:vp-Legendre_r} follows directly from Lemmas~\ref{lem:main_even} and \ref{lem:main_odd}. It further implies
Theorem~\ref{thm:vp-Legendre_explicit_p3} by setting $r = p$ and using Lemma~\ref{lem:v_odd}.

We now prove Theorem~\ref{thm:vp-Legendre_explicit_p2}.
    
\begin{proof}[Proof of Theorem~\ref{thm:vp-Legendre_explicit_p2}.]
Using Eqs.~\eqref{eq:Kummer_nnover2}, \eqref{eq:Sp_pna}, and \eqref{leg2} for $p=2$ and $a=0$, we have
$$
\nu_2(\binom{2m}{m}) = 2s_2(m) - s_2(2m) = s_2(2m) = 2m - v_2((2m)!).
$$
Combining this with 
Theorem~\ref{thm:vp-Legendre_r} for $r=p=2$, we have
$$
\nu_2(P_n(2)) =
\begin{cases}
    -v_2(n!) & \text{if $n$ is even}; \\
    -v_2((n-1)!) + 1 & \text{if $n$ is odd}.
\end{cases}
$$
It remains to note that for an odd $n$, $v_2((n-1)!)=v_2(n!)$, which then allows to combine the two cases as $\nu_2(P_n(2)) = (n\bmod 2)-v_2(n!)$.
\end{proof}

Next, we tackle Theorem~\ref{thm:main}.
\begin{proof}[Proof of Theorem~\ref{thm:main}.]
As the theorem statement stipulates, let $p \geq 3$ be a prime number and $f(n) := \nu_p(P_n(p))$.
To evaluate $f(pn + a)$ for $n\geq 0$ and $0\leq a < p$, we distinguish between four cases, according to the parities of $n$ and $a$.

\medskip\noindent\textit{Case 1:} both $n = 2m$ and $a=2b$ are even. 
Using Theorem~\ref{thm:vp-Legendre_explicit_p3} and Eq.~\eqref{eq:Sp_pna}, we get 
\begin{align}\label{eq:case1}
    f(2(pm + b)) 
    &= \frac{2 s_p(pm + b) - s_p(2pm + 2b)}{p-1}
    = \frac{2 (s_p(m) + b) - (s_p(2m) + 2b) }{p-1} \notag\\
    &= \frac{2 s_p(m) - s_p(2m)}{p-1}
    = f(2m).
\end{align}

\medskip\noindent\textit{Case 2:} $n = 2m$ is even and $a= 2b+1$ is odd. 
Since $\nu_p(2pm + 2b + 1)=0$, Theorem~\ref{thm:vp-Legendre_explicit_p3} and Eq.~\eqref{eq:case1} yield
\begin{align*}
    f(2(pm + b) + 1) 
    &= 1 + \nu_p(2pm + 2b + 1) + \nu_p( {2(pm + b) \choose pm + b} )\\
    & = 1 + \nu_p( {2(pm + b) \choose pm + b} ) 
    =1 + f(2(pm + b))  \\
    &= 1 + f(2m).
\end{align*}

\medskip\noindent\textit{Case 3:} $n= 2m+1$ is odd and $a=2b$ is even. 
By Theorem~\ref{thm:vp-Legendre_explicit_p3} and Eq.~\eqref{eq:Sp_pna}, we have
\begin{align*}
f(p(2m + 1) + 2b) 
&= \frac{2 s_p(pm + b + \frac{p-1}{2}) - s_p(p(2m + 1) + 2b) + p}{p-1} \\
&= 1 + \frac{2s_p(m) - s_p(2m+1) + p}{p-1} 
= 1 + f(2m+1).
\end{align*}

\medskip\noindent\textit{Case 4:} both $n=2m+1$ and $a=2b+1$ are odd. 
Since $b<\frac{p-1}2$, by Theorem~\ref{thm:vp-Legendre_explicit_p3} and Eq.~\eqref{eq:Sp_pna}, we have
\begin{align*}
    f(p(2m+1) + 2b+1)
    &= \frac{2 s_p(pm + b + \frac{p+1}{2}) - s_p(p(2m+1) + 2b+1)}{p-1} \\
    &= \frac{2s_p(m) - s_p(2m+1) + p}{p-1}     = f(2m+1).
\end{align*}
\end{proof}

Finally, we prove Theorem~\ref{cigler}, thus resolving the question of Cigler.
\begin{proof}[Proof of Theorem~\ref{cigler}.]
First, note that by substituting $x/(2-x)$ in Eq.~\eqref{eq:Rodrigues_Cigler} and comparing it with Eq.~\eqref{eq:Cigler_def}, we get
\begin{equation}\label{eq:MA}
    M_n(x) = (2-x)^n P_n\left(\frac{x}{2-x}\right).
\end{equation}
For a prime $p\geq 3$, Eq.~\eqref{eq:MA} implies
\[
    \nu_p(M_n(p)) = \nu_p(P_n\left(\frac{p}{2-p}\right)).
\]
Now set $r_1 = p$ and $r_2 = p/(p-2)$, and note that $\nu_p(r_1) = \nu_p(r_2) = 1$.
Then by Theorem~\ref{thm:vp-Legendre_r}, $\nu_p(P_n(r_2)) = \nu_p(P_n(r_1))$, implying that $\nu_p(M_n(p)) =  \nu_p(P_n(p))$.
\end{proof}

\bibliographystyle{habbrv}
\bibliography{refs}

\end{document}